\documentclass[11pt]{amsart}
\usepackage{}
\usepackage[T1]{fontenc}
\usepackage{lmodern}
\usepackage{ifthen}
\usepackage{amsfonts}
\usepackage{amsxtra}
\usepackage{amssymb}
\usepackage{amsthm}
\usepackage{array}
\usepackage[margin=1in]{geometry}
\usepackage{xcolor}
\definecolor{cite}{rgb}{0.50,0.00,1.00}
\definecolor{url}{rgb}{0.00,0.50,0.75}
\definecolor{link}{rgb}{0.00,0.00,0.50}
\usepackage[colorlinks,linkcolor=link,urlcolor=url,citecolor=cite,pagebackref,breaklinks]{hyperref}
\usepackage{mathtools}
\usepackage{mathrsfs}
\usepackage[all]{xy}
\usepackage[lite,abbrev,msc-links]{amsrefs}

\makeindex

\theoremstyle{definition} 
\newtheorem{Unity}{Unity}[section] 
\newtheorem*{defn*}{Definition} 
\newtheorem{defn}[Unity]{Definition} 

\theoremstyle{plain} 
\newtheorem*{thm*}{Theorem}
\newtheorem{thm}[Unity]{Theorem}
\newtheorem{prop}[Unity]{Proposition}
\newtheorem*{cor*}{Corollary}
\newtheorem{cor}[Unity]{Corollary}
\newtheorem{lem}[Unity]{Lemma}
\newtheorem{conj}[Unity]{Conjecture}

\theoremstyle{remark} 
\newtheorem*{rmk*}{Remark}
\newtheorem{rmk}[Unity]{Remark}

\numberwithin{Unity}{section}

\newcommand{\sH}{\mathscr{H}}

\newcommand{\sO}{\mathscr{O}}

\newcommand{\sX}{\mathscr{X}}

\newcommand{\bN}{\mathbb{N}}

\newcommand{\spec}{\textrm{Spec}}

\begin{document}
\title{On the Simultaneously Generation of Jets of the Adjoint Bundles}
\author[Junchao Shentu]{Junchao Shentu}
\email{stjc@ustc.edu.cn}
\address{School of Mathematical Sciences,
	University of Science and Technology of China, Hefei, 230026, China}
\author[Yongming Zhang]{Yongming Zhang}
\email{zhangyongming@amss.ac.cn}
\address{School of Mathematical Sciences, Fudan University, Shanghai, 200433,
	China}
\maketitle
\begin{abstract}
In this paper, we investigate the problem of simultaneously generations of $r$-jets of $\omega_X\otimes L^{\otimes m}$ when $L$ is ample and base point free. It turns out that in this case, the bound of $m$ is optimistic, i.e. $m\geq f(r)=n+r+1$, which is linear in $r$. Our results work over arbitrary characteristics. We also treat the same problem when $X$ is singular where $\omega_X$ is replaced by the pushforward of canonical sheaves or the cohomology sheaves of dualizing complexes.
\end{abstract}
\section{Introduction}
Let $X$ be a projective variety of dimension $n$ and $F$ be a coherent sheaf on $X$. Given $l$ closed points $x_i\in X$ and non-negative integers $r_i$, $i=1,\cdots,l$, we say that $F$ \textbf{simultaneously generates $r_i$-jets on $x_i$} if the canonical morphism
$$H^0(X,F)\rightarrow H^0(X,F\otimes\sO_X/m_{x_1}^{r_1+1}\cdots m_{x_l}^{r_l+1})$$
is surjective. $F$ \textbf{simultaneously generates $r$-jets} if for any integer $l\geq 1$ and $x_1,\cdots,x_l\in X$, $F$ simultaneously generates $r_i$-jets on $x_i$ whenever $\sum_{i=1}^l (r_i+1)=r+1$. When $F$ is a line bundle, this notion is equivalent to $r$-jets ampleness introduced by M. C. Beltrametti and A. J. Sommese (\cite{Beltrametti1993}).

If $F$ simultaneously generates $r$-jets, then for any 0-dimensional subscheme $Z$ of length $r+1$ (i.e. $\dim H^0(Z,\sO_Z)=r+1$), the natural morphism
$$H^0(X,F)\rightarrow H^0(X,F\otimes\sO_Z)$$
is surjective (Proposition \ref{prop_sepjet_imply_sep0cycles}).

There has been a great deal of interests for decades in exploring the generation of jets of the adjoint linear series $|\omega_X\otimes L|$ where $L$ is a line bundle on $X$. Among them there is a celebrate conjecture by Fujita about simultaneously generations of 0-jets and 1-jets.
\begin{conj}[Fujita \cite{Fujita1987}]
Let $X$ be a complex projective manifold of dimension $n$ and $L$ be an ample line bundle on $X$. Then
\begin{enumerate}
  \item $\omega_X\otimes L^{\otimes(n+1)}$ is globally generated;
  \item $\omega_X\otimes L^{\otimes(n+2)}$ is very ample.
\end{enumerate}
\end{conj}
This conjecture is proved up to dimension $3$ (\cite{Ein-Lazarsfeld1993,Reider1988}). The global generation part in dimension $4$ is proved by Kawamata (\cite{Kawamata1997}) and the affirmative answer in dimension $5$ is claimed by Fei Ye and Zhixian Zhu (\cite{Ye-Zhu2015}). It is known that $\omega_X\otimes L^{\otimes(n+1)}$ is always semiample (\cite{Kollar-Mori1998}, Theorem 3.3) and in particular is nef. This implies that $\omega_X\otimes L^{\otimes(n+2)}$ is ample.

For higher $r$-jets when $L$ is ample, Demailly (\cite{Demailly1993}) uses the method of the solution of the Monge-Ampr\`{e}re equation to obtain the generation of $r$-jets at every point of $X$ by global holomorphic sections of $\omega_X^{\otimes 2}\otimes L^{\otimes m}$ for $m\geq 6(n+r)^n$. Later Siu \cite{Siu1996} shows that $m\geq 2(n+2+n\binom{3n+2r-1}{n})$ is enough. He also gets results on generations of jets on 0-dimensional subsets: Given $l$ points $x_1,\cdots, x_l$ on $X$ and non-negative integers $r_1,\cdots, r_l$, $\omega_X^{\otimes 2}\otimes L^{\otimes m}$ simultaneously generates $r_i$-jets on $x_i$ whenever $m\geq 2(n+2+n\sum_{i=1}^l\binom{3n+2r_i-1}{n})$. This bound is improved by Demailly (\cite{Demailly1996}) to $m\geq 2+\sum_{i=1}^l\binom{3n+2r_i-1}{n}$.

In \cite{A-Siu1995}, Siu gives a criteria for separating distinct points of $\omega_X\otimes L^{\otimes m}$ instead of $ \omega_X^{\otimes2}\otimes L^{\otimes m}$ for a general $n$ and the order $m$ is only quadratic in $n$ instead of exponential:
\begin{thm}[\cite{A-Siu1995}]\label{thm_siu}
Let $\kappa$ be a positive number. If $(L^d\cdot W)^{\frac{1}{d}}\geq \frac{1}{2}n(n+2r-1)+\kappa$ for any irreducible subvariety $W$ of dimension $1\leq d\leq n$ in $X$, then $\omega_X\otimes L$ simultaneously generates $0$-jets on any set of $r$ distinct points $x_1,\cdots, x_r$ of X. In particular, if $m\geq \frac{1}{2}(n^2+2rn-n+2)$, then $\omega_X\otimes L^{\otimes m}$ simultaneously generates $0$-jets on any $r$ distinct points $x_1,\cdots, x_r$ of X.
\end{thm}
J. Koll\'{a}r \cite{Kollar1997} generalizes this criterion to klt pairs by algebraic methods. 

In the case of surfaces the results are more optimistic.
\begin{thm}[Beltramette and Sommese, \cite{Beltramette-Sommese1991}]
Let $L$ be a nef line bundle on a complex smooth projective surface $X$ and suppose that $L^2\geq4r+5$. Then either $\omega_X\otimes L$ is $r$-very ample or there exists an effective divisor $D$ containing some $0$-dimensional scheme of length $\leq r+1$ along which $r$-very ampleness fails, such that a power of the line bundle $L(-2D)$ has sections and
$$(D,L)-r-1\leq D^2<\frac{1}{2}(D,L)<r+1.$$
\end{thm}
By $r$-very ampleness we mean that for each 0-dimensional subscheme $Z$ of length $r+1$, the canonical morphism $H^0(\omega_X\otimes L)\rightarrow H^0(X,\omega_X\otimes L\otimes\sO_Z)$ is surjective. As a consequence (c.f. \cite{Beltrametti1993} Proposition 2.1), if $L$ is ample, then $\omega_X\otimes L^{\otimes m}$ simultaneously generates $r$-jets whenever $m\geq 3$ for $r=0$ and $m\geq (r+1)(r+2)$ for $r\geq 1$.

Extrapolating from the Fujita's conjecture, one may hope that $\omega_X\otimes L^{\otimes (n+r+1)}$ simultaneously generates $r$-jets if $L$ is ample. However, this is false even for the generation of $r$-jets on a single point. In fact, \cite{Ein1995} shows that there cannot exist a linear function $f(r)$ (depending on $n$ but independent of $X$ and $L$) such that $\omega_X\otimes L^{\otimes f(r)}$ generates $r$-jets at \textit{every} point $x\in X$. On the other side, it is shown in [loc.cit.] that if $L$ is ample and $m\geq n(n+r)$, then $\omega_X\otimes L^{\otimes m}$ generates $r$-jets on a \textit{general} point $x\in X$.

The story in positive characteristic is less known. Even the Fujita conjecture in dimension 2 is open. The positive results along this line are due to K. E. Smith and D. S. Keeler (\cite{Keeler2008,Smith1997}). They respectively show that Fujita's conjecture (1) and (2) holds over positive characteristics when $L$ is ample and globally generated .

The paper is organized as follows. The second section collects the main results in this paper. The third section
presents some preliminary materials which are used in the last section to prove the main results.

\textbf{Notations:}
\begin{itemize}
  \item A variety $X$ over a field $k$ is a reduced scheme finite type over $k$. We do not assume that $X$ is irreducible.
  \item Denote $\omega_{X/Y}$ be the relative dualizing sheaf when $X\rightarrow Y$ is a Cohen-Macaulay morphism between varieties. When $Y=\textrm{Spec}(k)$, we use $\omega_X$ for simplicity. For arbitrary variety $Z$ over $k$, denote by $\omega^\bullet_{Z/k}$ the dualizing complex of $Z$ over $k$.
\end{itemize}
\section{Main Results}
Let $k$ be an algebraically closed field and all varieties are defined over $k$. The main technical result in this paper is
\begin{thm}[Theorem \ref{thm_regular_sepajets}]\label{intro_thm_regular_sepajets}
Let $X$ be a projective variety and $L$ be an ample and globally generated line bundle on $X$. If $F$ is a $0$-regular coherent sheaf with respect to $L$, then $F\otimes L^{\otimes r}$ simultaneously generates  $r$-jets.
\end{thm}
A coherent sheaf $F$ on $X$ is $0$-regular with respect to $L$ if
$H^i(X,F\otimes L^{\otimes (-i)})=0,\ \forall i>0$.

As applications of this theorem, we prove the following Fujita-type results:
\begin{cor}\label{thm_Fujita}[Theorem \ref{thm_log_Fujita}]
Let $X$ be a projective smooth variety over $k$ ($\textrm{char}(k)=p\geq0$) and $D$ be a normal crossing divisor on $X$. Let $L$ be an ample and globally generated line bundle and $A$ be an ample line bundle on $X$, then
$\omega_X\otimes\sO_X(D)\otimes A\otimes L^{\otimes m}$ simultaneously generates  $r$-jets whenever $m\geq n+r$.
\end{cor}
The bound $n+r$ in the theorem is optimistic, as $A=L=\sO(1)$ on $\mathbb{P}^n$ ($D=\emptyset$) shows.

In characteristic 0, the cases when $r=0,1$ and $D$ is empty were known as a consequence of the Kodaira vanishing theorem (\cite{Lazarsfeld2004}, Example 1.8.23).

In characteristic $p>0$, our result is a generalization of of the Fujita type theorems on $0,1$-jets of K. E. Smith and D. S. Keeler (\cite{Keeler2008,Smith1997}).

\begin{cor}\label{cor_klt_Fujita}[Remark \ref{rmk_listof0regular} (3)]
Assume that $(X,\Delta)$ is a projective klt pair over $k$ with $\textrm{char}(k)=0$ such that $K_X+\Delta$ is numerically equivalent to a Cartier divisor. Let $L$ be an ample and globally generated line bundle, then
$\sO_X(K_X+\Delta)\otimes L^{\otimes m}$ simultaneously generates $r$-jets whenever $m\geq n+r+1$.
\end{cor}
Combined with Theorem \ref{thm_siu}, we are able to show the following effective results on the generation of jets.
\begin{cor}
Let $X$ be a complex projective manifold of dimension $n$ and $D$ be a normal crossing divisor. Let $L$ be an ample line bundle on $X$. Then for every $m\geq \frac{1}{2}(n^2+n+2)$, then $(\omega_X\otimes\sO_X(D)\otimes L^{\otimes m})^{\otimes (n+r+1)}$ generates simultaneous $r$-jets.
\end{cor}
\begin{cor}
Let $X$ be a complex projective manifold of dimension $n$ and $D$ be a normal crossing divisor. Suppose that the log canonical line bundle $\omega_X\otimes\sO_X(D)$ is ample, then $(\omega_X\otimes\sO_X(D))^{\otimes m}$ generates simultaneous $r$-jets (i.e. $r$-ample) whenever $m\geq\frac{1}{2}(n^2+n+4)(n+r+1)$.
\end{cor}

If $X$ is singular, there are two candidates for the canonical bundle: the cohomology sheaves of dualizing complex and the Grauert-Riemenschneider canonical sheaf.

\begin{cor}\label{cor_intro}
Let $X$ be an $n$-dimensional projective variety over $k$ and $L$ be an ample and globally generated line bundle on $X$. Let $\omega^\bullet_{X/k}$ be the dualizing complex of $X$. If one of the followings holds:
\begin{enumerate}
  \item $\textrm{char}(k)>0$ and $X$ is $F$-injective,
  \item $\textrm{char}(k)=0$ and $X$ is $F$-injective type,
\end{enumerate}
then for every integer $i$, $h^i(\omega^\bullet_{X/k})\otimes L^{\otimes m}$ simultaneously generates $r$-jets whenever $m\geq n+r+1$.
\end{cor}
When $\textrm{char}(k)>0$, $X$ is called $F$-injective if the trace morphism $F_\ast\omega^\bullet_{X/k}\rightarrow\omega^\bullet_{X/k}$ induces surjective morphisms on all the cohomology sheaves. When $\textrm{char}(k)=0$, $X$ is called $F$-injective type if there is a model $\sX$ of $X$ over $\spec(R)$ where $R$ is of finite type over $\mathbb{Z}$ such that infinite number of the closed fibers of $\sX$ are $F$-injective. In positive characteristic, $F$-injective singularity contains all the singularities ($F$-regular, $F$-split, $F$-rational, etc.) which have birational geometric interests. In characteristic 0, it is proved by K. Schwede \cite{Schwede2009} that $F$-injective type singularities are always du Bois, which are the natural singularities in the moduli problem of algebraic varieties. Whether the converse is true remains open. The readers can consult \cite{F-singular} for a survey of $F$-singularities.

Corollary \ref{cor_intro} has a potential application in the birational geometry of varieties with canonical singularities. Notice that when $X$ is normal, we have $$i_\ast\omega_{X_{\textrm{reg}}}\simeq h^{-\dim X}(\omega^\bullet_{X/k})$$ where $i:X_{\textrm{reg}}\rightarrow X$ is the canonical emersion. It is conjectured (c.f.\cite{Hara-Watanabe2002}, Problem 5.1.2) that log canonical singularities are always $F$-pure type (and hence $F$-injective type). If this is true, then for any projective variety $X$ with log canonical singularities and any ample and base point free line bundle $L$ on $X$, $\sO_X(K_X)\otimes L^{\otimes m}$ simultaneously generates $r$-jets whenever $m\geq n+r+1$. This is interesting because Kodaira vanishing theorem fails for log canonical varieties and Theorem \ref{intro_thm_regular_sepajets} does not work directly in this case.

\begin{cor}\label{thm_relative_Fujita}[Remark \ref{rmk_listof0regular} (2)]
Assume $\textrm{char}(k)=0$. Let $X$ be an $n$-dimensional projective variety over $k$ and $L$ be an ample and globally generated line bundle on $X$. Let $f:Y\rightarrow X$ be a proper morphism such that $Y$ is smooth. Then for each $j$,
$R^jf_\ast\omega_Y\otimes L^{\otimes m}$ simultaneously generates  $r$-jets when $m\geq n+r+1$. Moreover, $f_\ast(\omega_Y^{\otimes l})\otimes L^{\otimes m}$ simultaneously generates  $r$-jets when $m\geq l(n+1)+r$.
\end{cor}
Let $X$ be a variety over $k$ and $\mu:Y\rightarrow X$ be a resolution of singularities, then $\mu_\ast\omega_Y$ is independent of the choice of resolutions and is called the Grauert-Riemenschneider canonical sheaf. This sheaf can be also be constructed by taking a cubical resolution of $X$ (\cite{Peters-Steenbrink2008}).

This is the relative version of Theorem \ref{thm_Fujita}. One may wonder whether the globally generating assumption of $L$ can be removed:

Let $X$ be an $n$-dimensional projective variety over $k$ and $L$ be an ample line bundle on $X$. Let $f:Y\rightarrow X$ be a proper morphism where $Y$ is a smooth variety. Does
$R^jf_\ast\omega_Y\otimes L^{\otimes (n+r+1)}$ simultaneously generate  $r$-jets, at least for $r=0,1$?

However, this fails even for $r=0$ if there is no restriction on the singularity of $X$.
Let $X$ be a cubic nodal curve in $\mathbb{P}^2$ and $f: \mathbb{P}^1\rightarrow X$ be the normalization of $X$. Let $L=\sO(P)$ where $P$ is a smooth point on $X$. Consider the exact sequence
$$H^0(X,L)\rightarrow H^0(X,L\otimes k(P))\rightarrow H^1(X,\sO_X)\rightarrow H^1(X,L).$$
Since $H^1(X,L)=H^0(X,L^{-1}\otimes\omega_X)=0$, we see that $P$ is the base point of $L$. Since $$f_*\omega_{\mathbb{P}^1}\otimes L^{\otimes 2}=f_*(\omega_{\mathbb{P}^1}\otimes f^*L^{\otimes 2})=f_*(\omega_{\mathbb{P}^1}\otimes \sO_{\mathbb{P}^1}(2))=f_*(\sO_{\mathbb{P}^1}),$$
$H^0(X,f_*\omega_{\mathbb{P}^1}\otimes L^{\otimes 2})\simeq H^0(\mathbb{P}^1,\sO_{\mathbb{P}^1})$ is 1-dimensional. Therefore $f_*\omega_{\mathbb{P}^1}\otimes L^{\otimes 2}$ is not globally generated since the fiber of $f_*\omega_{\mathbb{P}^1}\otimes L^{\otimes 2}$ at the nodal point is 2-dimensional.


The right conjecture is proposed by Y. Kawamata \cite{Kawamata2000}:
\begin{conj}
Let $f:Y\rightarrow X$ be a surjective morphism from a smooth projective variety to a smooth projective variety of dimension $n$ such that $f$ is smooth over $X_0= X\setminus B$ for a normal crossing divisor $B$ on $X$. Let $L$ be
an ample divisor on $X$. Then the locally free sheaf $R^jf_\ast\omega_Y\otimes L^{\otimes m}$ is generated by global sections if $m\geq n+1$, or $m=n$ and the intersection number $(L^n)\geq 2$  for all $j\geq 0$.
\end{conj}
This conjecture is known for $\dim X\leq 4$ (\cite{Kawamata2000}).

Corollary \ref{thm_relative_Fujita} does not hold in positive characteristics. In fact, we show that (Proposition \ref{prop_contex_relative_Fujita}) given any $m\in\bN$, there is a curve fibration $X\rightarrow \mathbb{P}^1_k$ with the total space $X$ being a smooth surface, such that $f_\ast\omega_X\otimes \sO_{\mathbb{P}^1_k}(m)$ is not globally generated. The failure of Corollary \ref{thm_relative_Fujita} in positive characteristics is related to the pathological (non-positive) behaviour of the pushforward of canonical sheaves.

\section{Preliminary}
\begin{defn}[Castelnuovo-Mumford regularity]
Let $X$ be a projective variety and $L$ be an ample line bundle on $X$ which is generated by global sections. A coherent sheaf $F$ on $X$ is $m$-regular with respect to $L$ if
$$H^i(X,F\otimes L^{\otimes (m-i)})=0, \forall i>0.$$
\end{defn}

\begin{thm}[Mumford \cite{Mumford1966}]\label{thm_Mumford}
Let $F$ be a $0$-regular coherent sheaf on $X$ with respect to $L$,
then $F$ is generated by global sections.
\end{thm}
\begin{lem}\label{lem_general_section}
Let $X$ be an $n$-dimensional projective variety and $L$ be a globally generated ample line bundle on $X$. Let $x\in X$ be a closed point and let $Z=\bigcup Z_i$ be a subvariety of $X$ whose irreducible components $Z_i$ are of dimension  $\geq1$, then there is a global section $s\in H^0(X,L)$ such that its zero locus passes through $x$ and contains no $Z_i$.
\end{lem}
\begin{proof}
Since $L$ is ample and globally generated, there is a finite morphism $\phi:X\rightarrow\mathbb{P}(H^0(X,L))$ such that $L\simeq\phi^\ast\sO_{\mathbb{P}(H^0(X,L))}(1)$. Obviously the general hyperplanes $H$ of $\mathbb{P}(H^0(X,L))$ containing $\phi(x)$ does not contain any $\phi(Z_i)$ (whose dimensions are $\geq 1$). The pullback section $\phi^\ast(H)$ satisfies the requirement of the lemma.
\end{proof}
\subsection{Koszul Complex}

If $L^{\bullet}$ and $M^{\bullet}$ are two complexes of coherent sheaves, their tensor product $L^{\bullet}\otimes_{\sO_X}M^{\bullet}$ is, by definition, the complex such that $$(L^{\bullet}\otimes M^{\bullet})^n=\bigoplus_{p+q=n}L^{p}\otimes M^{q}$$ and such that $d:(L^{\bullet}\otimes  M^{\bullet})^n\rightarrow
(L^{\bullet}\otimes M^{\bullet})^{n+1}$ is defined on $L^{p}\otimes M^{q}$
by the formula $$d(x\otimes y)=d_{L}(x)\otimes y+(-1)^px\otimes d_{ M}(y).$$
Let $L$ be a line bundle on $X$ and $H^0(X,L)\neq0$.
A non-zero section $s\in \Gamma(X, L)$ defines a 2-term complex
$$L^{\bullet}(s): \cdots\rightarrow0\rightarrow L^{-1}\stackrel{s}\rightarrow\sO_X\rightarrow0\rightarrow \cdots$$
where $L^{0}(s)=\sO_X$ and $L^{-1}(s)=L^{-1}$.
For any complex $C^{\bullet}$, we put $$L^{\bullet}(s_1,\cdots,s_n, C^{\bullet})=C^{\bullet}\otimes L^{\bullet}(s_1)\otimes\cdots\otimes L^{\bullet}(s_n).$$
If $ F$ is an $\sO_X$-module we also view it as a complex by
$ F^n=0\ (n\neq0)$ and $ F^0= F$ and we put $$L^{\bullet}(s_1,\cdots,s_n, F)=
  F\otimes L^{\bullet}(s_1)\otimes\cdots\otimes L^{\bullet}(s_n).$$
These complexes are called Koszul complexes.

For any coherent sheaf $F$ the support of $F$ is the closed subset $\text{Supp}(F)=\{x\in X\mid E_x\neq0\}$. Its dimension is called the dimension of the sheaf $F$ and is denoted by $\text{dim}(F)$.
\begin{lem}\label{Koszul complex}
Let $X$ be an $n$-dimensional projective variety and $L$ be a globally generated ample line bundle on $X$. Let $x\in X$ be a closed point. Then for any coherent sheaf $ F$ there is a sequence of global sections $s_1,\cdots,s_n \in \Gamma(X, L)$ whose common zero locus is a zero dimensional subscheme $Z$ containing $x$, and such that the cohomology sheaves of the Koszul complex
$L^{\bullet}(s_1,\cdots,s_n, F)$ are zero dimensional, i.e.
$$\text{dim}(h^i (L^{\bullet}(s_1,\cdots,s_n, F)))=0\text{ for all }i.$$
\end{lem}
\begin{proof}
If $\text{dim}( F)=0$, any sections $s_1,\cdots,s_n\in \Gamma(X, L)$ whose zero locuses intersect properly and contain $x$ will do. We assume that $\text{dim}( F)>0$. Since $ F$ is coherent, there are finitely many associate points of $ F$. By Lemma \ref{lem_general_section}, one can choose a global
sections $s_1\in H^0(X,L)$ whose zero locus $Z_1= Z(s_1)$ passes through $x$ but does not contain any
non-closed associate point of $ F$. We have the exact sequence of complexes
$$0\rightarrow\sO_X\rightarrow L^{\bullet}(s_1)\rightarrow L^{-1}[1]\rightarrow0.$$
Tensoring the above exact sequence with $F$, we get
$$0\rightarrow F\rightarrow L^{\bullet}(s_1, F)\rightarrow F\otimes L^{-1}[1]\rightarrow0,$$
and a long exact sequence
$$0\rightarrow h^{-1}(L^{\bullet}(s_1, F))\rightarrow h^0( F)\otimes L^{-1}\stackrel{\delta^0}{\rightarrow} h^0( F)\rightarrow h^0(L^{\bullet}(s_1, F))\rightarrow0.$$
The connecting homomorphism $\delta^0$ is the multiplication by $s_1$.
Hence $\dim (\ker(\delta^0))=0$. Therefore
$$\text{dim}(h^{i} (L^{\bullet}(s_1,  F)))=0,\quad i\neq0$$
and
$$h^0(L^{\bullet}(s_1,  F))= F/s_1 F.$$
If $\text{dim}(h^0 (L^{\bullet}(s_1, F)))=0$ then we finish the proof.

Now suppose that we already have a Koszul complex $L^{\bullet}(s_1,\cdots,s_m, F)$ such that
$$\text{dim}(h^{i} (L^{\bullet}(s_1,\cdots,s_m,  F)))=0,\quad i\neq0$$
and
$$h^0 (L^{\bullet}(s_1,\cdots,s_m,  F))= F/(s_1,\cdots,s_m) F.$$
If $\text{dim}(h^0 (L^{\bullet}(s_1,\cdots,s_m,  F)))\neq0$, by the similar method as above we can choose a global section $s_{m+1}\in H^0(X,L)$ whose zero locus $Z_{m+1}= Z(s_{m+1})$ passes through $x$ but does not contain any non-closed associate points of $h^0(L^{\bullet}(s_1,\cdots,s_m, F))$. Tensoring $L^{\bullet}(s_1,\cdots,s_m, F)$ with the exact sequence $$0\rightarrow\sO_X\rightarrow L^{\bullet}(s_{m+1})\rightarrow L^{-1}[1]\rightarrow0,$$
we get
$$0\rightarrow L^{\bullet}(s_1,\cdots,s_m, F)\rightarrow
L^{\bullet}(s_1,\cdots,s_{m+1}, F)\rightarrow
L^{\bullet}(s_1,\cdots,s_m, F)\otimes L^{-1}[1]\rightarrow0.$$
This induces a long exact sequence
$$\cdots\stackrel{\delta^{-p-1}}\rightarrow h^{-p-1}(L^{\bullet}(s_1,\cdots,s_m, F))\rightarrow
h^{-p-1}(L^{\bullet}(s_1,\cdots,s_{m+1}, F))\rightarrow
h^{-p}(L^{\bullet}(s_1,\cdots,s_m, F))(L^{-1})\stackrel{\delta^{-p}}\rightarrow$$
$$\cdots\stackrel{\delta^{-1}}\rightarrow h^{-1}(L^{\bullet}(s_1,\cdots,s_m, F))\rightarrow
h^{-1}(L^{\bullet}(s_1,\cdots,s_{m+1}, F))\rightarrow
h^{0}(L^{\bullet}(s_1,\cdots,s_m, F))(L^{-1})\stackrel{\delta^0}{\rightarrow}$$
$$ h^{0}(L^{\bullet}(s_1,\cdots,s_m, F))\rightarrow h^{0}(L^{\bullet}(s_1,\cdots,s_{m+1}, F))\rightarrow0.$$
Note that the connecting homomorphism $\delta^{-p}$ is the multiplication by $(-1)^p s_{m+1}$. By induction we obtain a Koszul complex $L^{\bullet}(s_1,\cdots,s_{m+1}, F)$ such that
$$\text{dim}(h^{i} (L^{\bullet}(s_1,\cdots,s_{m+1},  F)))=0,\quad i\neq0$$
and
$$h^0 (L^{\bullet}(s_1,\cdots,s_{m+1},  F))= F/(s_1,\cdots,s_{m+1}) F.$$
After finite steps we get $s_1,\cdots,s_r\in \Gamma(X, L)$ ($r\leq n$) such that
$$\dim (F/(s_1,\cdots,s_{r}) F)=0,$$
and by the construction,
$$\text{dim}(h^{i} (L^{\bullet}(s_1,\cdots,s_{r},  F)))=0 \text{ for all }i.$$
If $r=n$, then we finish the proof.
If $r<n$, we can choose general sections $s_{r+1},\cdots,s_n$ and finish the proof.
\end{proof}

\section{Regularity and Generation of Jets}
Throughout this section, $k$ is an algebraically closed field in arbitrary characteristic unless otherwise specified. All varieties are defined over $k$. First we recall the following
\begin{defn}
Let $X$ be a proper variety. A coherent sheaf $F$ on $X$ simultaneously generates  $r$-jets if for any $r+1$ closed points $x_1,\cdots,x_{r+1}\in X$ (not necessarily different from each other), the canonical morphism
$$H^0(X,F)\rightarrow H^0(X,F\otimes\sO_X/m_{X,x_1}m_{X,x_2}\cdots m_{X,x_{r+1}})$$
is surjective.
\end{defn}
When $F$ is a line bundle, this notion is equivalent to the $k$-jet ampleness in \cite{Beltrametti1993}. Simultaneously generating $r$-jets implies the separation of jets along any 0-dimensional subschemes of length $r+1$, as the following proposition shows.
\begin{prop}\label{prop_sepjet_imply_sep0cycles}
If $F$ simultaneously generates  $r$-jets, then for any 0-dimensional subscheme $Z$ of length $r+1$ (i.e. $\dim_k H^0(Z,\sO_Z)=r+1$), the natural morphism
$$H^0(X,F)\rightarrow H^0(X,F\otimes\sO_Z)$$
is surjective.
\end{prop}
\begin{proof}
Suppose that $$\sO_Z=\bigoplus_{i=1}^{m}\sO_{X}/I_i$$
where $\{x_i\}_{i=1,\cdots,m}$ are closed points of $X$ and for each $1\leq i\leq m$, $\sqrt{I_i}=m_{x_i}$. Denote $r_i=\dim_k \sO_{X}/I_i$ and $\sum_{i=1}^m r_i=r+1$, then by Nakayama's lemma we have $(m_{x_i}/I_i)^{r_i}=0$, which implies $m_{x_i}^{r_i}\subseteq I_i$. Therefore there is a natural surjective morphism
$$H^0(X,F\otimes\sO_X/m^{r_1}_{x_1}\cdots m^{r_m}_{x_{m}})\rightarrow H^0(X,F\otimes\sO_Z)$$
Hence we get the surjective morphism
$$H^0(X,F)\rightarrow H^0(X,F\otimes\sO_Z).$$
\end{proof}

\begin{thm}\label{thm_regular}
Let $X$ be a projective variety and $L$ be an ample and globally generated line bundle on $X$. If $F$ is a $0$-regular coherent sheaf with respect to $L$, then $F\otimes L\otimes m_{x}$ is $0$-regular with respect to $L$ for every closed point $x\in X$.
\end{thm}
\begin{proof}
By Lemma \ref{Koszul complex} there is a sequence of global sections $s_1,\cdots,s_n \in \Gamma(X, \mathscr{L})$ whose common zero locus is a 0-dimensional closed subscheme $Z$ containing $x$, such that the cohomology sheaves of the Koszul complex
$\mathscr{L}^{\bullet}(s_1,\cdots,s_n,\mathscr{F})$ are 0-dimensional, i.e.
$$\text{dim}(h^i (\mathscr{L}^{\bullet}(s_1,\cdots,s_n,\mathscr{F})))=0\text{ for all }i.$$
Denote $I_Z$ be the ideal sheaf of $Z$.

Tensoring $L$ with the Kuszul complex $\mathscr{L}^{\bullet}(s_1,\cdots,s_n,\mathscr{F})$, we get
$$0\rightarrow F\otimes\bigwedge^n(L^{-1})^{\oplus n}\otimes L\rightarrow F\otimes\bigwedge^{n-1}(L^{-1})^{\oplus n}\otimes L\rightarrow\cdots
\rightarrow F\otimes (L^{-1})^{\oplus n}\otimes L\stackrel{d^{-1}}\rightarrow I_Z\cdot\mathscr{F}\otimes L$$
where $d^{-1}$ is surjective and all the cohomology sheaves are 0-dimensional.
Since $\mathscr{F}$ is $0$-regular, $F\otimes\bigwedge^i(L^{-1})^{\oplus n}\otimes L$ is $i-1$-regular for every $i\geq 1$. Therefore $I_Z\cdot\mathscr{F}\otimes L$ is $0$-regular by an argument on chasing through the above long exact sequence.

Note that $(m_x\cdot\mathcal {F}/I_Z\cdot\mathcal {F})\otimes\mathscr{L}$ is a subsheaf of $h^0(\mathscr{L}^{\bullet}(s_1,\cdots,s_n,F)\otimes L$ which is of dimension zero. Therefore $(m_x\cdot\mathcal {F}/I_Z\cdot\mathcal {F})\otimes\mathscr{L}$ is 0-dimensional and hence is $0$-regular. As a consequence, $m_x\cdot F\otimes L$ is $0$-regular by chasing the following exact sequence.
$$0\rightarrow I_Z\cdot F\otimes L\rightarrow m_x\cdot F\otimes L\rightarrow(m_x\cdot\mathcal {F}/I_Z\cdot\mathcal {F})\otimes L\rightarrow0$$

Since the kernel of $m_x\otimes F\otimes L \twoheadrightarrow m_x\cdot F\otimes L$ is 0-dimensional, we see that $m_x\otimes F\otimes L$ is $0$-regular.
\end{proof}
\begin{lem}\label{lem_separate_jets}
Let $X$ be a projective variety and $L$ be an ample and globally generated line bundle on $X$. Assume that $r$ is a non-negative integer and $F$ is a coherent sheaf on $X$. If for every closed point $(x_1,\cdots,x_l)\in X^l$ with $0\leq l\leq r$, $m_{x_1}\cdots m_{x_l}\cdot F\otimes L^{\otimes l}$ is generated by global sections, then $F\otimes L^{\otimes r}$ simultaneously generates  $s$-jets for every $0\leq s\leq r$.
\end{lem}
\begin{proof}
We prove the lemma by induction on $s$.

The case when $s=0$ is trivial. Let $x_0,\cdots,x_s$ be $s+1$ closed points of $X$. Consider the following diagram
$$\xymatrix{
0 \ar[d] & 0\ar[d]\\
H^0(X,m_{x_0}\cdots m_{x_{s-1}}\cdot F\otimes L^{\otimes r}) \ar[r]^-{\varphi}\ar[d] & H^0(X,(m_{x_0}\cdots m_{x_{s-1}}\cdot F/m_{x_0}\cdots m_{x_{s}}\cdot F)\otimes L^{\otimes r}) \ar[d]\\
H^0(X,F\otimes L^{\otimes r}) \ar[r]^-{\phi_s}\ar[d]^{\phi_{s-1}} & H^0(X,F\otimes L^{\otimes r}\otimes \sO_X/m_{x_0}\cdots m_{x_{s}}) \ar[d]^{\psi}\\
H^0(X,F\otimes L^{\otimes r}\otimes \sO_X/m_{x_0}\cdots m_{x_{s-1}})\ar[r]^-{Id} & H^0(X,F\otimes L^{\otimes r}\otimes \sO_X/m_{x_0}\cdots m_{x_{s-1}})\\
}$$
By induction, $\phi_{s-1}$ is surjective  and then $\psi$ is  surjective. Therefore the two collum complexes are short exact sequences. Moreover the horizontal morphism $\varphi$ is equivalent to the canonical morphism
$$H^0(X,m_{x_0}\cdots m_{x_{s-1}}\cdot F\otimes L^{\otimes r})\rightarrow H^0(X,m_{x_0}\cdots m_{x_{s-1}}\cdot F\otimes L^{\otimes r}\otimes k(x_s)),$$
which is surjective since $m_{x_0}\cdots m_{x_{s-1}}\cdot F\otimes L^{\otimes s}$ and $L$ are generated by global sections. Then the surjectivity of $\phi_s$ follows from the snake lemma.
\end{proof}
\begin{thm}\label{thm_regular_sepajets}
Let $X$ be a projective variety and $L$ be an ample and globally generated line bundle on $X$. If $F$ is a $0$-regular coherent sheaf with respect to $L$, then $F\otimes L^{\otimes r}$ simultaneously generates  $r$-jets.
\end{thm}
\begin{proof}
By Theorem \ref{thm_regular}, for every $l$ closed points $x_1,\cdots,x_l\in X$ with $l\leq r$, $$F\otimes(L\otimes m_{x_1})\otimes\cdots\otimes(L\otimes m_{x_l})$$ is $0$-regular. Since the kernel of the natural surjective morphism
$$F\otimes(L\otimes m_{x_1})\otimes\cdots\otimes(L\otimes m_{x_l})\twoheadrightarrow m_{x_1}\cdots m_{x_l}\cdot F\otimes L^{\otimes l}$$
is 0-dimensional, $m_{x_1}\cdots m_{x_l}\cdot F\otimes L^{\otimes l}$ is 0-regular. By Theorem \ref{thm_Mumford}, it is generated by global sections. Then the Theorem follows from Lemma \ref{lem_separate_jets}.
\end{proof}
\begin{rmk}\label{rmk_listof0regular}
0-regularity of coherent sheaves always comes from various vanishing theorems. We make a non-complete list of 0-regular coherent sheaves as follows:
\begin{enumerate}
  \item $F=\omega_X\otimes\sO_X(D)$ where $X$ is an $n$-dimensional smooth projective variety over $k$ with $\textrm{char}(k)=0$ and $D$ is a normal crossing divisor on $X$. Let $A$ be an ample line bundle. As a consequence of Norimatsu's vanishing theorem (\cite{Norimatsu1978}),  $K_X\otimes\sO_X(D)\otimes L^{\otimes n}\otimes A$ is 0-regular with respect to any globally generated ample line bundle $L$. This proves Corollary \ref{thm_Fujita} in the case of $\textrm{char}(k)=0$.
  \item $F=R^jf_\ast \omega_Y$, $j\geq 0$ where $f:Y\rightarrow X$ is a proper morphism over $k$ with $\textrm{char}(k)=0$ such that $Y$ is smooth. Then $R^jf_\ast \omega_Y\otimes L^{\otimes \dim X+1}$ is 0-regular with respect to any globally generated ample line bundle $L$ on $X$. This is a consequence of the Koll\'{a}r's vanishing theorem (\cite{Kollar1986}).
      $$H^i(X,R^jf_\ast \omega_Y\otimes H)=0,\text{ for all } i>0,j\geq 0$$
      where $H$ is an ample line bundle.
      In particular, when $f$ is birational, $f_\ast\omega_Y$ coincides with $F^n\widetilde{\Omega}_X^\bullet$ where $\widetilde{\Omega}_X^\bullet$ is the filtered de Rham complex of $X$ defined by taking a cubical hyper-resolution of $X$ (\cite{Peters-Steenbrink2008}). Moreover, Popa and Schnell (\cite{Popa-Schnell2014}) generalize Koll\'{a}r's vanishing theorem to the pushforward of pluri-canonical sheaf, namely
      $$H^i(X,f_\ast\omega_Y^{\otimes l}\otimes H^{\otimes m})=0,\text{ for all } m\geq l(n+1)-n, i>0$$
      where $H$ is an ample and globally generated line bundle. This implies that
      $$f_\ast\omega_Y^{\otimes l}\otimes L^{\otimes l(n+1)}$$
      is 0-regular with respect to $L$. These vanishing theorems lead to Corollary \ref{thm_relative_Fujita}.
  \item $F=\sO_X(K_X+\Delta)$ where $(X,\Delta)$ is a projective klt pair defined  over $k$ with $\textrm{char}(k)=0$ such that $K_X+\Delta$ is numerically equivalent to a Cartier divisor. By the Kawamata-Viehweg vanishing theorem (\cite{Kollar-Mori1998}), we have
      $$H^i(X,\sO_X(K_X+\Delta)\otimes L)=0,\text{ for all } i>0$$
      This leads to Corollary \ref{cor_klt_Fujita}.
\end{enumerate}
\end{rmk}
\subsection{Generation of Jets in Positive Characteristic}
Now we consider generations of jets in positive characteristics. Although the Kodaira-type vanishing fails generally over a positive characteristic field, there is another kind of coherent sheaves ensuring the results as in Theorem \ref{thm_regular_sepajets}.
\begin{thm}\label{thm_Frobenius_sepajets}
Let $X$ be an $n$-dimensional projective variety over $k$ with $\textrm{char}(k)=p>0$ and $L$ be an ample and globally generated line bundle on $X$. Let $A$ be an ample line bundle on $X$. Denote by $F_X:X\rightarrow X$ the absolute Frobenius morphism of $X$. Let $G$ be a coherent sheaf such that there exists a positive integer $N$ and a surjective morphism $\textrm{tr}:F_{X\ast}^{N}G\rightarrow G$. Then $G\otimes L^{\otimes (n+r)}\otimes A$ simultaneously generates $r$-jets.
\end{thm}
\begin{proof}
First we claim that there is a positive integer $N_0$ such that for any integer $N'\geq N_0$, $F_{X\ast}^{N'}G\otimes L^{\otimes n}\otimes A$ is $0$-regular with respect to $L$. In fact, this follows from Serre's vanishing theorem and the following projection formula
$$H^i(X,F_{X\ast}^{N'}G\otimes L^{\otimes (n-i)}\otimes A)\simeq H^i(X,G\otimes L^{\otimes p^{N'}(n-i)}\otimes A^{\otimes p^{N'}}).$$
By Theorem \ref{thm_regular}, for any non-negative integer $r$ and every closed point $(x_1,\cdots,x_l)\in X^l$ with $0\leq l\leq r$, $F_{X\ast}^{N'}G\otimes L^{\otimes (n+l)}\otimes A\otimes m_{x_1}\otimes\cdots \otimes m_{x_l}$ is generated by global sections. Since the kernel of the surjective morphism
$$F_{X\ast}^{N'}G\otimes L^{\otimes (n+l)}\otimes A\otimes m_{x_1}\otimes\cdots \otimes m_{x_l}\rightarrow
m_{x_1}\cdots m_{x_l}\cdot F_{X\ast}^{N'}G\otimes L^{\otimes (n+l)}\otimes A$$
has supports contained in the set of isolated points $\{x_1,\cdots,x_l\}$, we see that $m_{x_1}\cdots m_{x_l}\cdot F_{X\ast}^{N'}G\otimes L^{\otimes (n+l)}\otimes A$ is 0-regular. By Theorem \ref{thm_Mumford}, it is globally generated.

Now consider the composition of
$$F_{X\ast}^{kN}G\stackrel{F_{X\ast}^{(k-1)N}\textrm{tr}}\rightarrow \cdots\rightarrow F_{X\ast}^{2N}G\stackrel{F_{X\ast}^{N}\textrm{tr}}\rightarrow F_{X\ast}^{N}G\stackrel{\textrm{tr}}\rightarrow G$$
where $k$ is an integer such that $kN\geq N_0$, we get a surjective morphism $F_{X\ast}^{kN}G\rightarrow G$. As a consequence, for any non-negative integer $r$ and every closed point $(x_1,\cdots,x_l)\in X^l$ with $0\leq l\leq r$, $m_{x_1}\cdots m_{x_l}\cdot G\otimes L^{\otimes (n+l)}\otimes A$ is generated by global sections. Then the theorem follows from Lemma \ref{lem_separate_jets}.
\end{proof}
\begin{thm}\label{thm_log_Fujita}
Let $X$ be an $n$-dimensional smooth projective variety over $k$ ($\textrm{char}(k)=p\geq0$) and $D$ be a normal crossing divisor on $X$. Let $L$ be an ample and globally generated line bundle and $A$ be an ample line bundle on $X$. Then $\omega_X\otimes\sO_X(D)\otimes L^{\otimes (n+r)}\otimes A$ simultaneously generates  $r$-jets.
\end{thm}
\begin{proof}
If $\textrm{char}(k)=0$, this is a consequence of Theorem \ref{thm_regular_sepajets} (see Remark \ref{rmk_listof0regular} (1)). If $\textrm{char}(k)=p>0$, by Theorem \ref{thm_Frobenius_sepajets}
it suffices to construct a surjective morphism $F_\ast(\omega_X\otimes\sO_X(D))\rightarrow\omega_X\otimes\sO_X(D)$. We choose (\'etale) locally on $X$ parameters $x_1,\cdots,x_n$ where $D$ is defined by $x_1\cdots x_r=0$. Then the Cartier morphism
$$C:F_\ast(\omega_X)\rightarrow\omega_X$$
$$x_1^{p-1}\cdots x_n^{p-1}dx_1\wedge\cdots\wedge dx_n\mapsto dx_1\wedge\cdots\wedge dx_n$$
extends to a morphism
$$C':F_\ast(\omega_X\otimes\sO_X(D))\rightarrow\omega_X\otimes\sO_X(D)$$
$$f\frac{dx_1}{x_1}\wedge\cdots\wedge \frac{dx_r}{x_r}\wedge dx_{r+1}\wedge\cdots\wedge dx_n\mapsto f^{[\frac{1}{p}]}\frac{dx_1}{x_1}\wedge\cdots\wedge \frac{dx_r}{x_r}\wedge dx_{r+1}\wedge\cdots\wedge dx_n.$$
Here $f^{[\frac{1}{p}]}$ is the unique function such that $f=(f^{[\frac{1}{p}]})^p+g$, where $g$ does not contains any terms like $(x_1^{a_1}\cdots x_k^{a_k})^p$. It is easy to see that $C'$ is surjective.
\end{proof}
Recall that for a projective variety $X$ over $k$, the non-trivial cohomology sheaves of the dualizing complex $\omega^\bullet_{X/k}$ lie in $[-\dim X,0]$. If $\textrm{char}(k)=p>0$, by applying $R\sH om_{\sO_X}(-,\omega^\bullet_{X/k})$ on the canonical morphism
$$\sO_X\rightarrow F_{X\ast}\sO_X,$$
we gets the trace morphism
$$\textrm{tr}: F_{X\ast} \omega^\bullet_{X/k}\rightarrow \omega^\bullet_{X/k}.$$
$X$ is called $F$-injective if
$$\textrm{tr}^i:F_{X\ast}h^i(\omega^\bullet_{X/k})\rightarrow h^i(\omega^\bullet_{X/k})$$
is surjective for all $i$. When $\textrm{char}(k)=0$, $X$ is called $F$-injective type if there is a model $\sX$ of $X$ over $\spec(R)$ where $R$ is of finite type over $\mathbb{Z}$ such that infinite number of the closed fibers of $\sX$ are $F$-injective.
It is proved by K. Schwede \cite{Schwede2009} that $F$-injective type singularities are always du Bois. Whether the converse is true remains open.
\begin{cor}
Let $X$ be an $n$-dimensional projective variety over $k$ and $L$ be an ample and globally generated line bundle on $X$. Let $\omega^\bullet_{X/k}$ be the dualizing complex of $X$. If one of the followings holds:
\begin{enumerate}
  \item $\textrm{char}(k)>0$ and $X$ is $F$-injective,
  \item $\textrm{char}(k)=0$ and $X$ is $F$-injective type,
\end{enumerate}
then  $h^i(\omega^\bullet_{X/k})\otimes L^{\otimes (n+r+1)}$ simultaneously generates $r$-jets.
\end{cor}
\begin{proof}
(1) follows from Theorem \ref{thm_Frobenius_sepajets}. (2) follows from the standard argument of reduction to mod $p$.
\end{proof}
\begin{rmk}
It is conjectured that log canonical singularities are always $F$-pure type (and hence $F$-injective type). If this is true, then we see that if $X$ is a log canonical projective variety, then $\sO_X(K_X)\otimes L^{\otimes (n+r+1)}$ simultaneously generates  $r$-jets, provided that $L$ is ample and globally generated. This is interesting because Kodaira vanishing theorem fails for log canonical varieties.
\end{rmk}
In the end we give an example showing that Corollary \ref{thm_relative_Fujita} fails in positive characteristics.
\begin{prop}\label{prop_contex_relative_Fujita}
Let $k$ be an algebraically closed field with $\textrm{char}(k)=p>0$, then for any $m\in\bN$, there is a curve fibration $X\rightarrow \mathbb{P}^1_k$ with the total space $X$ being a smooth surface, such that $f_\ast\omega_X\otimes \sO(m)$ is not globally generated.
\end{prop}
\begin{proof}
By \cite{Bailly1981} there is a semi-stable surface fibration $g:X^\prime \rightarrow \mathbb{P}^1_k$ over $k$ such that$$g_*\omega_{X^\prime/\mathbb{P}^1_k}\backsimeq\sO_{\mathbb{P}^1_k}(-1)\oplus\sO_{\mathbb{P}^1_k}(p).$$
Consider the Frobenius base change of this fibration
$$\xymatrix{
X\ar[dr]_f \ar[r]^h &X^{\prime\prime} \ar[d]_-p\ar[r]^q &X^\prime\ar[d]^-g\\
 & \mathbb{P}^1_k\ar[r]^{F^n} &\mathbb{P}^1_k
}$$
where $F^n$ is the $n$ times compositions of the absolute Frobenius morphism of $\mathbb{P}^1_k$ and $h$ is the minimal desingularization of the normal surface $X^{\prime\prime}=\mathbb{P}^1_k\times_{\mathbb{P}^1_k}X^\prime$.
Since $g$ is a semi-stable fibration, $X^{\prime\prime}$ has only rational double point singularities. So it has only Gorenstein singularities and
$$h_*\omega_{X}=\omega_{X^{\prime\prime}}.$$
Since $g$ is a CM morphism, we have the following relation (\cite{Conrad2000}) between the relative dualizing sheaves:
$$\omega_{X^{\prime\prime}/\mathbb{P}^1_k}=q^*\omega_{X^{\prime}/\mathbb{P}^1_k}.$$
Consequently we have:
\begin{eqnarray*}
f_*\omega_X&=&p_*h_*\omega_{X}=p_*\omega_{X^{\prime\prime}}
=p_*\omega_{X^{\prime\prime}/\mathbb{P}^1_k}\otimes\omega_{\mathbb{P}^1_k}
=p_*q^*\omega_{X^{\prime}/\mathbb{P}^1_k}\otimes\omega_{\mathbb{P}^1_k}\\
&=&(F^n)^*g_*\omega_{X^{\prime}/\mathbb{P}^1_k}\otimes\omega_{\mathbb{P}^1_k}
=(F^n)^*(\sO_{\mathbb{P}^1_k}(-1)\oplus\sO_{\mathbb{P}^1_k}(p))\otimes\sO_{\mathbb{P}^1_k}(-2)\\
&=&\sO_{\mathbb{P}^1_k}(-p^n-2)\oplus\sO_{\mathbb{P}^1_k}(p^{n+1}-2).
\end{eqnarray*}
For any line bundle $L=\sO_{\mathbb{P}^1_k}(m)$ on $\mathbb{P}^1_k$, we can choose a $n$ such that $-p^n-2+2m<0$. Then
$f_*\omega_X\otimes L^{\otimes 2}=\sO_{\mathbb{P}^1_k}(-p^n-2+2m)\oplus\sO_{\mathbb{P}^1_k}(p^{n+1}-2+2m)$
always has a direct summand of negative degree which breaks the global generation of $f_*\omega_X\otimes L^{\otimes 2}$.
%
\end{proof}

\end{document}